\documentclass[10pt]{article}
\usepackage{mathrsfs}
\usepackage{amsfonts}
\usepackage{amsthm,amsmath,amssymb,anysize}
\newtheorem{lemma}{Lemma}[section]
\newtheorem{theorem}[lemma]{Theorem}
\newtheorem{remark}[lemma]{Remark}
\newtheorem{proposition}[lemma]{Proposition}

 \usepackage[numbers,sort&compress]{natbib}
\marginsize{3.8cm}{3.8cm}{3.6cm}{3cm}
\numberwithin{equation}{section}

\title{\textsf{Filtration, automorphisms and classification of the infinite dimensional  odd Contact
superalgebras}}
\author{\textsc{Jixia Yuan$^{1,}$}\footnote{Correspondence:
 jxy@hrbnu.edu.cn (J. Yuan), wendeliu@ustc.edu.cn} \footnote{Supported by  the NSF of HLJ
  Province (A200903)}\;  \textsc{and Wende Liu$^{2,}$}
 \footnote{Supported by the NSF (10871057)
  of China and the NSF  (A200802) of HLJ Province, China} \\
  \\
  \textit{$^{1}$Department of Mathematics},
  \textit{Harbin Institute of Technology}\\
  \textit{Harbin 150006, China}\\
\\
  \ \ \textit{$^{2}$School of Mathematical Sciences},
  \textit{Harbin Normal University} \\
  \textit{Harbin 150025, China}
  }
\date{ }
\begin{document}
\maketitle

\begin{quotation}
\noindent\textbf{Abstract}:   The principal  filtration of the infinite-dimensional odd Contact Lie superalgebra over
a field of characteristic $p>2$  is proved to be invariant under the automorphism group by investigating ad-nilpotent elements and determining certain invariants such as subalgebras generated by some ad-nilpotent elements. Then, it is proved that two automorphisms coincide  if and only if they coincide on
 the  $-1$  component  with respect to the principal grading. Finally, all the odd Contact superalgebras are   classified up to isomorphisms. \\


  \end{quotation}

  \setcounter{section}{-1}
\section{Introduction}

\noindent As is well known, filtration techniques are of great importance in the structure
and classification theories of Lie (super)algebras.
A descending filtration of a Lie superalgebra $L$
 is a sequence of $\mathbb{Z}_{2}$-graded spaces $L=L_0\supset L_1\supset\cdots$ for which
 $[L_i, L_j]\subset L_{i+j}$ holds for all $i,j.$ In the situation that
 the Killing form on a simple Lie (super)algebra is degenerate,
 the filtration structure plays a
particular role. A filtration $L=L_0\supset L_1\supset\cdots$ of a Lie superalgebra $L$ is said to be invariant provided that
 $\varphi(L_i)\subset L_i$
for all $i$ and all automorphisms  $\varphi$ of $L$. We know that the simple Lie (super)algebras of
 Cartan type possess various natural filtration structures, for which
the invariance may be used to make an insight for the  intrinsic
properties and the automorphism groups of those Lie (super)algebras.
The filtration structures have been studied  for the finite dimensional Lie
superalgebras of Cartan type, for example, in \cite{lz1,zf,zn}
the invariance of the natural filtrations  was determined for the
generalized Witt superalgebras, the special superalgebras, the
Hamiltonian superalgbras and the odd Hamiltonian superalgebras. The
reader is also refereed to \cite{lz2,zl5} for the similar work on certain
infinite dimensional modular Lie superalgebras of Cartan type.
Let us state the main results of this paper. Write $KO(n,n+1)$
for
  the infinite dimensional
  odd Contact Lie superalgebras over a field of prime characteristic  (see Sec.1
for  a definition and more details).  $KO(n,n+1)$ has a canonical filtration structure known as principal. By means of
characterizing ad-nilpotent elements of $KO(n,n+1)$ we first obtain in this paper that:
\begin{itemize}
\item (Theorem \ref{t15}) \textsf{The principal filtration  of the   odd Contact  Lie
superalgebra is invariant under the automorphism group of the Lie
superalgebra.}
\end{itemize}
As a consequence the automorphisms can be characterized as follows:
\begin{itemize}
\item   (Theorem \ref{t17}) \textsf{Two
automorphisms of the odd Contact Lie superalgebra coincide  if and only if they coincide on
 the $-1$ component   with respect to the principal
grading.}
\end{itemize}
Finally we classify all the infinite dimensional
  odd Contact Lie superalgebras up to isomorphisms:
\begin{itemize}
\item (Theorem \ref{t16}) \textsf{$KO(n,n+1)\cong KO(m,m+1)$ if and only if $n=m.$}
\end{itemize}

\section{Preliminaries}

Throughout $\mathbb{F}$ is a field of characteristic $p>2$;
$\mathbb{Z}_2:= \{\bar{0},\bar{1}\}$ is the additive group of two
elements; $\mathbb{N}$ and $\mathbb{N}_0$ are the sets of positive
integers and nonnegative integers, respectively. Let
$\mathcal{O}(n)$ be the divided power algebra   with
$\mathbb{F}$-basis $\{x^{(\alpha)}\mid \alpha\in
\mathbb{N}_{0}^{n}\}$. Note that $x^{(0)}:=1\in\mathcal{O}(n)$,
where $0=(0,\ldots,0)\in \mathbb{N}_{0}^{n}.$ For
$\varepsilon_i:=(\delta_{i1},\delta_{i2},\ldots,\delta_{in})\in
\mathbb{N}_{0}^{n}$, write $x_i$ for $x^{(\varepsilon_i)}$, where
$i=1,\ldots,n.$ Let $\Lambda(m)$ be the exterior superalgebra over
$\mathbb{F}$ in $m$ variables $x_{n+1},x_{n+2},\ldots,x_{n+m}$.
 Set
$$\mathbb{B}(m):=\left\{ \langle i_1,i_2,\ldots,i_k\rangle \mid n+1\leq i_1<i_2<\cdots <i_k\leq n+m,\;
k\in \overline{0, m}\right\}.$$
 For
 $u:=\langle i_1,i_2,\ldots,i_k\rangle \in \mathbb{B}(m),$ write $ |u| :=k$
 and  $x^u:=x_{i_1}x_{i_2}\cdots x_{i_k}.$
  Notice that we also denote the index set $\{i_1,i_2,\ldots,i_k\}$ by $u$ itself.
  For $u, v \in \mathbb{B}(m)$ with $u\cap
   v=\emptyset,$ define $u+ v$ to be the uniquely determined element
  $w\in \mathbb{B}(m)$ such that $ w=u\cup
  v.$ If $v\subset u,$ define $u-v$ to be $w\in
  \mathbb{B}(m)$ such that $w=u\setminus v.$
  Clearly, the associative superalgebra
$\mathcal{O}(n,m):=\mathcal{O}(n)\otimes \Lambda(m)$ has a so-called
\textit{standard} $\mathbb{F}$-\textit{basis} $\{ x^{(\alpha )
}x^u\mid (\alpha,u) \in \mathbb{N}_{0}^{n}\times \mathbb{B}(m)\}.$

 Let
$\partial_r$ be the superderivations of $\mathcal{O}(n,m) $ defined
by $\partial_{r}(x^{(\alpha)})=x^{(\alpha-\varepsilon_{r})}$ for
$r\in \overline{1,n} $ and $\partial_{r}(x_{s})=\delta_{rs}$ for
$r,s\in \overline{1, n+m}.$ Here  $\overline{1,n}$ is the set of integers $1,2,\ldots,n.$ The  generalized Witt superalgebra  $
W\left(n,m\right)$ is $\mathbb{F}$-spanned by all $f_r
\partial_r,$ where $ f_r\in \mathcal{O}(n,m),$ $r\in
\overline{1,n+m}.$   Note that $W(n,m) $ is a free $ \mathcal{O}
\left(n,m\right)$-module with basis $ \{
\partial_r\mid r\in
 \overline{1,n+m}\}.$ In particular, $W (n,m ) $ has a
\textit{standard} $\mathbb F$-\textit{basis}
 $\{x^{(\alpha)}x^{u}\partial_r\mid (\alpha,u,r)\in
\mathbb{N}_{0}^{n}\times\mathbb{B}(m)\times\overline{1,n+m}\}.$

 For an $n$-tuple
$\alpha:=(\alpha_1,\ldots,\alpha_n)\in \mathbb{N}_0^n$, put
$|\alpha|:=\sum_{i=1}^n\alpha_i.$ When $m=n+1$, we usually write
$\mathcal {O}:=\mathcal {O}(n,n+1),$ $W:= W(n,n+1),$
$\mathbb{A}:=\mathbb{N}_{0}^{n}$ and $\mathbb{B}:=\mathbb{B}(n+1).$
If $u:=\langle i_{1},\ldots, i_{r}\rangle\in \mathbb{B},$ put
\[\parallel u\parallel:= \left\{\begin{array}{ll}
|u|+1, &\mbox{if}\; 2n+1\in u
\\|u|, &\mbox{if}\; 2n+1\notin u.
\end{array}\right.\]
 Recall the   standard $\mathbb{Z}$-grading,
$ \mathcal{O}=\oplus_{i\geq 0}\mathcal{O}_{\mathbf{s},[i]},$ where
$$\mathcal{O}_{\mathbf{s},[i]}:={\rm
span}_{\mathbb{F}}\{x^{(\alpha)}x^u\mid |\alpha|+|u|=i, \alpha\in
\mathbb{A}, u\in \mathbb{B}\}.$$
It induces naturally the   standard
grading $ W=\oplus_{i\geq -1}W_{\mathbf{s},[i]},$ where
$$W_{\mathbf{s},[i]}:={\rm span}_{\mathbb{F}}\{f\partial_j\mid
f\in{\mathcal{O}_{\mathbf{s},[i+1]}}, j\in \overline{1,2n+1}\}.$$
The standard gradings of $ \mathcal{O}$ and $W$ are  of type
$(1,\ldots,1\mid 1,\ldots,1).$ Let $W_{\mathbf{s},i}=\sum_{j\geq
i}W_{\mathbf{s},[j]}.$ Then $(W_{\mathbf{s},i})_{i\geq-1}$ is called
the standard filtration of $W.$

We shall also use the   principal grading $\mathcal{O}=\oplus_{i\geq
0}\mathcal{O}_{\mathbf{p},[i]},$ where
  $$
 \mathcal{O}_{\mathbf{p},[i]}:={\rm
span}_\mathbb{F}\{x^{(\alpha)}x^u\mid |\alpha|+\|u\|=i,\,\alpha\in
\mathbb{A},\, u\in \mathbb{B}\},$$ and the   principal grading
$W=\oplus_{i\geq-2} W_{\mathbf{p},[i]},$ where
$$ W_{\mathbf{p},[i]}:={\rm
span}_{\mathbb{F}}\{f\partial_j\mid
f\in{\mathcal{O}_{\mathbf{p},[i+1+\delta_{j,2n+1}]}},j\in
\overline{1,2n+1}\}$$
 (cf. \cite{k2}).  The principal gradings of $
\mathcal{O}$ and $W$ are   of type
$(1,\ldots,1\mid 1,\ldots,1,2). $

For a   vector superspace $V=V_{\bar{0}}\oplus V_{\bar{1}}, $ we
write $\mathrm{p}(x):=\theta$ for the  parity  of a $\mathbb{Z}_{2}$-homogeneous
element $x\in V_{\theta}, $ $\theta\in \mathbb{Z}_{2}.$ Once the
symbol $\mathrm{p}(x)$ appears, it
 will imply that $x$ is  a $\mathbb{Z}_2$-homogeneous element.

 The odd Contact superalgebra, which is  a subalgebra of $W,$ is
defined as follows (see \cite{k2} for more details):
$$
KO(n,n+1):=\{D_{KO}(a)\mid a\in \mathcal{O}\},
$$
where
$$D_{KO}(a):=T_{H}(a)+(-1)^{\mathrm{p}(a)}\partial_{2n+1}(a)E+ (E(a)-2a )\partial_{2n+1},$$
$$E:=\sum _{i=1}^{2n}x_{i}\partial_{i},\quad T_{H}(a):=\sum
_{i=1}^{2n}(-1)^{\mu(i')\mathrm{p}(a)}\partial_{i'}(a)\partial_{i},$$
\[i':= \left\{\begin{array}{ll}
i+n, &\mbox{if}\; i\in\overline{1, n}
\\i-n, &\mbox{if}\; i\in \overline{n+1, 2n},
\end{array}\right.\quad
\mu(i):=\left\{\begin{array}{ll}\bar{0},  &\mbox{if}\; i\in \overline{1,n}\\
\bar{1},&\mbox{if}\; i\in\overline{n+1,2n+1}.
\end{array}\right. \]
For the operator $T_{H}$ and further information, the reader is referred to \cite{lzw}. Note that  for $a,b\in \mathcal{O}$ (see \cite{k2}),
\begin{equation}\label{liuee1}
[D_{KO}(a),
D_{KO}(b)]=D_{KO}(D_{KO}(a)(b)-(-1)^{\mathrm{p}(a)}2\partial_{2n+1}(a)b).
\end{equation}

 For simplicity, we usually write $KO$ for $KO(n,n+1)$. Note that
 $KO $ has  a so-called principal
$\mathbb{Z}$-grading structure denoted by
$$
KO =\oplus_{i\geq-2}KO _{\mathbf{p},[i]},\quad\mbox{where
$
KO _{\mathbf{p},[i]}:=KO \cap W_{\mathbf{p},[i]}.
$}
$$
In particular,
\begin{eqnarray}\label{e2}
&& KO _{\mathbf{p},[-2]}=\mathbb{F}\cdot D_{KO}(1),\nonumber\\
&&KO _{\mathbf{p},[-1]}={\rm span}_{\mathbb{F}}\{D_{KO}(x_i)\mid
i\in\overline{1,2n}\},\nonumber\\
&& KO _{\mathbf{p},[0]}={\rm span}_{\mathbb{F}}\{D_{KO}(x_{2n+1}),
D_{KO}(x_ix_j)\mid 1\leq i\leq j\leq 2n\}.
\end{eqnarray}
Let
$$X_{\mathbf{p},i}:=\sum_{j\geq i}X_{\mathbf{p},[j]},  \quad\mbox{where $X=W$
or $KO.$}$$  Then $(X_{\mathbf{p},i})_{i\geq-2}$ is called the
principal filtration of $X.$
Recall that  the infinite-dimensional generalized Witt superalgebra
$W(n,n)$ contains the following Lie superalgebra as a subalgebra (see \cite{k2,lh}):
\begin{eqnarray*}
 SHO'(n,n):=\{T_{H}(a)\mid a\in \mathcal
{O}(n,n), \Delta(a)=0\},\quad\mbox{where $\Delta:=\sum_{i=1}^{n}\partial_{i}\partial_{i'}.$}
\end{eqnarray*}

\noindent\textbf{Convention}: In the sequel we shall write $KO _{[i]}$ and
$KO _{i}$ for $KO _{\mathbf{p},[i]}$ and $KO _{\mathbf{p},i},$
respectively.
\section{Ad-nilpotent elements}
Let $L$ be a Lie superalgebra.
An element $y\in L$ is  $\mathrm{ad}$-nilpotent if as a transformation $\mathrm{ad}y$ is nilpotent on $L$.
 Let $R$ be a subalgebra of $L$. Put
\begin{eqnarray*}
&&\mathrm{nil}(R):=\mbox{the set of the elements in $R$ which are ad-nilpotent on $L$},\\
&&\mathrm{span}_{\mathbb{F}}\mathrm{nil}(R):=\mbox{the subspace spanned by $\mathrm{nil}(R)$},\\
&&\mathrm{Nil}(R):=\mbox{the
subalgebra of $L$ generated by $\mathrm{nil}(R).$}
\end{eqnarray*}

\begin{lemma}\label{l1'} Suppose $a\in \mathcal{O}$ is of $\mathbb{Z}$-degree $2$ with respect to the standard grading.
If $\partial_{2n+1}(a)=0$,  that is, if $a\in
\mathcal {O}(n,n)$, then
\begin{equation*}
   [D_{KO}(a),D_{KO}(b)]=D_{KO}(T_{H}(a)(b))\quad \mbox{for all}\;b\in \mathcal{O}.
\end{equation*}
\end{lemma}
\begin{proof}
It follows from $(\ref{liuee1}).$
\end{proof}

A nonempty subset $S$ of a Lie superalgebra $L$ is called a Lie-super subset if  $S$ is
closed under the multiplication of $L $ and it spans a  sub-superspace (and then is a Lie superalgebra).
A slight modification of \cite[Theorem 1.3.1]{sf} yields the following lemma.

\begin{lemma}\label{l1'''}
Suppose $V$ is a  vector superspace over $\mathbb{F}$ and
$S$ a Lie-super subset of  Lie superalgebra $\mathfrak{gl}(V)$. If
$S$ consists of nilpotent linear transformations
$\mathrm{span}_{\mathbb{F}}S$ is of finite dimension, then $\mathrm{span}_{\mathbb{F}}S$ is
strictly   triangulable  on $V,$ that is, there is a finite sequence $(V_i)_{0\leq i\leq m}$ of  sub-superspaces
such that
$$0 = V_0\subset V_1\subset \cdots \subset V_m = V;\quad x(V_i)\subset V_{i-1}\quad \mbox{for all}\; x\in \mathrm{span}_{\mathbb{F}}S.$$
\end{lemma}
\begin{lemma}\label{l1''}
$W_{\mathbf{p},1}\subseteq \mathrm{nil}(W).$
\end{lemma}
\begin{proof}
There exists a sub-superspace $V\subseteq W_{\mathbf{s},1}$ such that
\begin{eqnarray*}
W_{\mathbf{p},1}=\mathrm{span}_{\mathbb{F}}\{x_{2n+1}\partial_{i}\mid
i\in \overline{1,2n}\}+V.
\end{eqnarray*}
 For $E\in W_{\mathbf{p},1},$ write $E=E_{0}+E_{1},$ where
$E_{0}\in \mathrm{span}_{\mathbb{F}}\{x_{2n+1}\partial_{i}\mid i\in
\overline{1,2n}\},$  $E_{1}\in V.$ Put $E_{2}:=[E_{0},E_{1}].$ Then
there exists  an $n$-tuple $\underline{t} $ of positive integers   such that $E_{1}, E_{2}\in
W_{\mathbf{s},1}(n,n+1;\underline{t}):=W(n,n+1;\underline{t})\cap W_{\mathbf{s},1} $ [see \cite{lzw} for a
definition of $W(n,n+1;\underline{t})$].
Note that
\begin{eqnarray*}S:=\mathrm{span}_{\mathbb{F}}\{\mathrm{ad}(x_{2n+1}\partial_{i})\mid
i\in \overline{1,2n}\}\cup \mathrm{ad}W_{\mathbf{s},1}(n,n+1;\underline{t})
\end{eqnarray*}
  is a Lie-super subset of  the general linear Lie superalgebra $\mathfrak{gl}(W).$  A direct computation shows
that $\mathrm{span}_{\mathbb{F}}\{ x_{2n+1}\partial_{i}\mid
i\in \overline{1,2n}\}\subseteq\mathrm{nil}(W).$ By virtue of \cite[Theorem 2.5]{zl5},
we have $ W_{\mathbf{s},1}(n,n+1;\underline{t})\subseteq
\mathrm{nil}(W).$ Then Lemma \ref{l1'''} ensures that   $E\in
\mathrm{nil}(W).$
\end{proof}
\begin{lemma}\label{l1}
$\mathrm{(1)}$ $KO _{1}\subseteq \mathrm{nil}(KO ).$

$\mathrm{(2)}$ Suppose $y=y_{[i]}+y_{i+1}\in \mathrm{nil}(KO _{i}),$
where $y_{[i]}\in KO _{[i]}$ and $y_{i+1}\in KO _{i+1}.$ Then
$y_{[i]}\in \mathrm{nil}(KO _{[i]}).$

$\mathrm{(3)}$ Suppose  $y=y_{[-1]}+y_{0}\in \mathrm{nil}(KO
_{\bar{0}}),$ where   $y_{[-1]}\in KO _{[-1]}\cap KO _{\bar{0}}$ and
$y_{0}\in KO _{0}\cap KO _{\bar{0}}.$ Then $y_{[-1]}=0.$

$\mathrm{(4)}$ $\mathrm{Nil}(KO _{\bar{0}})=\mathrm{Nil}(KO
_{[0]}\cap KO _{\bar{0}})+KO _{1}\cap KO _{\bar{0}}.$
\end{lemma}
\begin{proof}
(1) Since $KO _{1}\subseteq W_{\mathbf{p},1},$ by virtue of Lemma
\ref{l1''} we get $KO _{1}\subseteq \mathrm{nil}(KO ).$

(2) Similar to \cite[Lemma 2.7]{zl5}, one may prove (2).

(3) Suppose $y_{[-1]}=\sum_{i=1}^{n}a_{i}D_{KO}(x_{i'}),$ where
$a_{i}\in \mathbb{F}.$  Note that
$D_{KO}(x^{((k+1)\varepsilon_{j})})\in KO $ for $k\in \mathbb{N}.$
If $a_{j}\neq0$ for some $j\in \overline{1,n},$ then
 $y_{[-1]}$ is not  $\mathrm{ad}$-nilpotent, since
$(\mathrm{ad}y_{[-1]})^{k}(D_{KO}(x^{((k+1)\varepsilon_{j})}))=a_{j}^{k}D_{KO}(x_{j})\neq0$. This
contradicts (2). Therefore, $a_{j}=0$ for all $j\in
\overline{1,n},$ that is, $y_{[-1]}=0.$

(4) It follows from (1) and (3).
\end{proof}
\begin{lemma}\label{l3}
If $i\neq j'\in \overline{1,2n},$ then $(T_{H}(x_{i}x_{j}))^{2p}=0.$
\end{lemma}
\begin{proof}
Note that
\begin{eqnarray}\label{e3}
T_{H}(x_{i}x_{j})=(-1)^{\mu(i)+\mu(i)\mu(j)}x_{j}\partial_{i'}
+(-1)^{\mu(j)}x_{i}\partial_{j'},
\end{eqnarray}
$(x_{j}\partial_{i'})^{p}=(x_{i}\partial_{j'})^{p}=0$ and
$[x_{j}\partial_{i'}, x_{i}\partial_{j'}]=0$ for all $i\neq j'\in
\overline{1,2n}.$ In combination with (\ref{e3}), we have
$(T_{H}(x_{i}x_{j}))^{2p}=0.$
\end{proof}
\begin{lemma}\label{l2}
For $i, j \in \overline{1,2n}, $ the following statements hold.

$\mathrm{(1)}$ If $f\in \mathcal {O},$ $D_{KO}(f)\in KO _{[0]}$ and
$\partial_{2n+1}(f)\neq0,$ then $D_{KO}(f)\notin \mathrm{nil}(KO
_{[0]}).$

$\mathrm{(2)}$ Suppose $i\neq j' $ and $a_{i}\in \mathbb{F} $ for all $i\in \overline{1,n}.$ Then
\begin{eqnarray*}
&& D_{KO}(x_{i}x_{j})\in \mathrm{nil}(KO _{[0]}),
\quad\qquad\quad\;\; \sum_{i=1}^{n}a_{i}D_{KO}(x_{i}x_{i'})\notin
\mathrm{nil}(KO _{[0]})\; \mbox{or is } 0,
\\
&& D_{KO}(x_{i}x_{i'})\notin \mathrm{Nil}( KO _{[0]}),
\quad\qquad\quad\; D_{KO}(x_{2n+1})\notin \mathrm{Nil}(KO _{[0]}),
\\
&& D_{KO}(x_{i}x_{i'}-x_{j}x_{j'})\in \mathrm{Nil}( KO _{[0]}).
\end{eqnarray*}
\end{lemma}
\begin{proof}

(1) Suppose  $f\in \mathcal {O},$ $D_{KO}(f)\in KO _{[0]}$ and
$\partial_{2n+1}(f)\neq0.$ Then there exist $0\neq a\in \mathbb{F}$
and  $f_{0}\in \mathcal {O}(n,n)$ such that $f=ax_{2n+1}+f_{0}.$
Since
\begin{eqnarray*}
[D_{KO}(ax_{2n+1}+f_{0}), D_{KO}(1)]=2aD_{KO}(1),
\end{eqnarray*}
we have $D_{KO}(f)\notin \mathrm{nil}(KO _{[0]}).$

(2) Applying Lemma \ref{l1'} we obtain by induction on $k$ that
$$(\mathrm{ad}D_{KO}(x_{i}x_{j}))^{k}(D_{KO}(f))=D_{KO}(T_{H}(x_{i}x_{j})^{k}(f))\;\mbox{for all} \; k\in \mathbb{N},$$
where $f\in \mathcal {O}.$  Since $\mathrm{Ker}(D_{KO})=0,$ one sees
that $D_{KO}(x_{i}x_{j})$ is $\mathrm{ad}$-nilpotent if and only if
$T_{H}(x_{i}x_{j})$ is a nilpotent transformation of $\mathcal {O}.$
Then by Lemma \ref{l3}, we know that
\begin{eqnarray*}
D_{KO}(x_{i}x_{j})\in \mathrm{nil}(KO _{[0]})\quad \mbox{for all}\;
i\neq j'\in \overline{1,2n}.
\end{eqnarray*}
Note that
\begin{eqnarray*}
\Big[\sum_{i=1}^{n}a_{i}D_{KO}(x_{i}x_{i'}), D_{KO}(x_{j}x_{2n+1})\Big]=-a_{j}D_{KO}(x_{j}x_{2n+1}),
\; j\in \overline{1,n}.
\end{eqnarray*}
If $\sum_{i=1}^{n}a_{i}D_{KO}(x_{i}x_{i'})\neq 0$ then
\begin{eqnarray*}
 \sum_{i=1}^{n}a_{i}D_{KO}(x_{i}x_{i'})\notin \mathrm{nil}(KO _{[0]}).
\end{eqnarray*}
It follows from (1)  that $\mathrm{nil}(KO _{[0]})\subseteq
SHO'(n,n). $ Therefore, $\mathrm{Nil}(KO _{[0]}) \subseteq SHO'(n,n).$ But $D_{KO}(x_{i}x_{i'}),$ $D_{KO}(x_{2n+1})\notin
SHO'(n,n),$ thus
$$D_{KO}(x_{i}x_{i'}), D_{KO}(x_{2n+1})\notin
\mathrm{Nil}(KO _{[0]}).$$ By virtue of the fact
that
\begin{eqnarray*}
[D_{KO}(x_{i}x_{j}), D_{KO}(x_{i'}x_{j'})]
=-D_{KO}(x_{i}x_{i'}-x_{j}x_{j'}),
\end{eqnarray*}
we have $ D_{KO}(x_{i}x_{i'}-x_{j}x_{j'})\in \mathrm{Nil}( KO
_{[0]}). $
\end{proof}
\section{Invariant subalgebras}
Let
\begin{eqnarray*}
&&\mathfrak{T}:=\mathrm{Nor}_{KO _{\bar{0}}}(\mathrm{Nil}(KO _{\bar{0}})),\\
&& \mathfrak{Q}:= \{y\in KO _{\bar{1}}\mid [y,KO _{\bar{1}}]\subseteq T\},
\\
&&\mathfrak{M}:=\{y\in KO _{\bar{1}}\mid [y,\mathfrak{Q}]\subseteq\mathrm{Nil}(KO
_{\bar{0}})\}.
\end{eqnarray*}
It is easy to see that $\mathfrak{T}$ is an invariant subspace under the
automorphisms  of $KO$ and so are $\mathfrak{Q}$ and $\mathfrak{M}$.
\begin{proposition}\label{p12}
$\mathfrak{T}=KO _{0}\cap KO _{\bar{0}}.$ In particular, $KO _{0}\cap KO
_{\bar{0}}$ is an invariant subalgebra of $KO $.
\end{proposition}
\begin{proof}
Let $$y=\sum_{i=1}^{n}a_{i}D_{KO}(x_{i'})+y''\in \mathfrak{T},$$ where
$a_{i}\in \mathbb{F}$ for all $i\in \overline{1,n},$ $y''\in KO
_{0}\cap KO _{\bar{0}}.$ Assume that $a_{j}\neq0$ for some $j\in
\overline{1,n}.$ Take $ j\neq k\in \overline{1,n}.$ By Lemma
\ref{l1}(4), we have $D_{KO}(x_{j}x_{k}x_{k'})\in KO _{1}\cap KO
_{\bar{0}}\subseteq \mathrm{Nil}(KO _{\bar{0}}).$ Then
\begin{eqnarray*}
&&-a_{j}D_{KO}(x_{k}x_{k'})-a_{k}D_{KO}(x_{j}x_{k'})+h\\
&=&\Big[\sum_{i=1}^{n}a_{i}D_{KO}(x_{i'})+y'',D_{KO}(x_{j}x_{k}x_{k'})\Big]
\in\mathrm{Nil}(KO _{\bar{0}}),
\end{eqnarray*}
where $h\in KO _{1}\cap KO _{\bar{0}}.$ This contradicts Lemma
\ref{l2}(2) and then $\mathfrak{T}\subseteq KO _{0}\cap KO _{\bar{0}}.$
On the other hand, by (\ref{e2}) and  Lemma \ref{l2}(2), we have
\begin{eqnarray*}KO _{0}\cap KO _{\bar{0}}=\mathrm{span}_{\mathbb{F}}\mathrm{nil}(KO _{\bar{0}})
+Y,
\end{eqnarray*}
where $Y:=\mathrm{span}_{\mathbb{F}}\{D_{KO}(x_{i}x_{i'}),
D_{KO}(x_{2n+1})\mid i\in \overline{1,n}\}.$ By (\ref{liuee1}), we
have $$ [Y, \mathrm{Nil}(KO _{\bar{0}})]\subseteq \mathrm{Nil}(KO
_{\bar{0}}). $$ Then
\begin{eqnarray*}
[KO _{0}\cap KO _{\bar{0}},\mathrm{Nil}(KO
_{\bar{0}})]=[\mathrm{span}_{\mathbb{F}}\mathrm{nil}(KO _{\bar{0}})
+Y, \mathrm{Nil}(KO _{\bar{0}})]\subseteq \mathrm{Nil}(KO
_{\bar{0}}).
\end{eqnarray*}  The proof is complete.
\end{proof}
\begin{lemma}\label{p5}
$\mathfrak{Q}\subseteq\mathrm{span}_{\mathbb{F}}\{D_{KO}(x_{i}x_{j})\mid 1\leq
i\leq j \leq n\}+KO _{1}\cap KO _{\bar{1}}.$
\end{lemma}
\begin{proof}
For   $y\in \mathfrak{Q},$ we may write $$y=D_{KO}(a)+y',$$ where $a\in
\mathbb{F},$ $y'\in KO _{-1}\cap KO _{\bar{1}}.$ Note that
$D_{KO}(x_{1'}x_{2n+1})\in KO _{\bar{1}}. $ Then
\begin{eqnarray*}
-2aD_{KO}(x_{1'})+h=[D_{KO}(a)+y',D_{KO}(x_{1'}x_{2n+1})]\in KO
_{0}\cap KO _{\bar{0}},
\end{eqnarray*}
where $h\in KO _{0}\cap KO _{\bar{0}}.$ Then $a=0.$ Thus we may
write
$$y=\sum_{i=1}^{n}a_{i}D_{KO}(x_{i})+y'',$$ where $a_{i}\in
\mathbb{F}$ for all $i\in \overline{1,n},$ $y''\in KO _{0}\cap KO
_{\bar{1}}.$ If $a_{j}\neq0$ for some $j\in \overline{1,n},$ take
$j\neq k \in \overline{1,n}.$ Note that $D_{KO}(x_{j'}x_{k'})\in KO
_{\bar{1}}.$  We have
\begin{eqnarray*}D_{KO}(a_{j}x_{k'}-a_{k}x_{j'})+h=\Big[\sum_{i=1}^{n}a_{i}D_{KO}(x_{i})
+y'',D_{KO}(x_{j'}x_{k'})\Big]\in KO _{0}\cap KO _{\bar{0}},
\end{eqnarray*}
where $h\in KO _{0}\cap KO _{\bar{0}},$ contradicting that
$a_{j}\neq 0$. Thus we may write
$$y=\sum_{1\leq i\leq j\leq n} a_{ij}D_{KO}(x_{i}x_{j})+\sum_{1\leq i< j\leq n}b_{ij}D_{KO}(x_{i'}x_{j'})+y'',$$
where $a_{ij}, b_{ij}\in \mathbb{F}$ for all $1\leq i\leq j\leq n,$
$y''\in KO _{1}\cap KO _{\bar{1}}.$ Assume that $b_{kl}\neq0$ for
some  $1\leq k< l\leq n.$ Since $D_{KO}(x_{k})\in KO _{\bar{1}},$
we have
\begin{eqnarray*}
\sum_{i=1}^{k-1}-b_{ik}D_{KO}(x_{i'})+\sum_{i=
k+1}^{n}b_{ki}D_{KO}(x_{i'})+h=[y,D_{KO}(x_{k})]\in KO _{0}\cap KO
_{\bar{0}},
\end{eqnarray*}
where $h\in KO _{0}\cap KO _{\bar{0}},$ contradicting the assumption
that $b_{kl}\neq0$. Then
$$\mathfrak{Q}\subseteq\mathrm{span}_{\mathbb{F}}\{D_{KO}(x_{i}x_{j})\mid
1\leq i\leq j \leq n\}+KO _{1}\cap KO _{\bar{1}}$$ and this
completes the proof.
\end{proof}

\begin{remark}\label{r6} For  $i\neq j \in \overline{1,n},$
we have
 $D_{KO}(x_{i}x_{i'}x_{2n+1})\in
\mathfrak{Q},
 D_{KO}(x_{i'}x_{j}x_{j'})\in \mathfrak{Q}.$

\end{remark}
\begin{proposition}\label{p7}
$\mathfrak{M}=KO _{0}\cap KO _{\bar{1}}.$ In particular, $KO _{0}\cap KO
_{\bar{1}}$ is an invariant subalgebra of $KO $.
\end{proposition}
\begin{proof}
By (\ref{e2}), Lemmas \ref{l1}(4), \ref{l2}(2) and
\ref{p5}, we have
\begin{eqnarray*}
[KO _{0} \cap  KO _{\bar{1}},\mathfrak{Q}] &\subseteq&  [KO _{0}\cap KO
_{\bar{1}},\mathrm{span}_{\mathbb{F}}\{D_{KO}(x_{i}x_{j})\mid 1\leq
i\leq j \leq n\}  +KO _{1}\cap
KO _{\bar{1}}]\\
&\subseteq& \mathrm{Nil}(KO _{\bar{0}}).
\end{eqnarray*}
Hence $KO _{0}\cap KO _{\bar{1}}\subseteq \mathfrak{M}.$ Conversely, for $y\in
\mathfrak{M},$ we may write $$y=D_{KO}(a)+y',$$ where $a\in \mathbb{F},$ $y'\in
KO _{-1}\cap KO _{\bar{1}}.$ By Remark \ref{r6}, we have
\begin{eqnarray*}
&&2aD_{KO}(x_{i}x_{i'})+h=[D_{KO}(a)+y',D_{KO}(x_{i}x_{i'}x_{2n+1})]\in
\mathrm{Nil}(KO _{\bar{0}}),
\end{eqnarray*}
where $h\in KO _{1}\cap KO _{\bar{0}}.$ By Lemma \ref{l2}(2), we
have $a=0.$ Thus we may write
$$y=\sum_{i=1}^{n}a_{i}D_{KO}(x_{i})+y'',$$ where $a_{i}\in
\mathbb{F}$ for all $i\in \overline{1,n},$ $y''\in KO _{0}\cap KO
_{\bar{1}}.$ Assume that $a_{j}\neq0$ for some $j\in
\overline{1,n}.$ Take $j\neq k\in \overline{1,n}$. By Remark
\ref{r6}, we have
\begin{eqnarray*}
&&a_{j}D_{KO}(x_{k}x_{k'})-a_{k}D_{KO}(x_{j'}x_{k})+h\\
&=&\Big[\sum_{i=1}^{n}a_{i}D_{KO}(x_{i})+y'',D_{KO}(x_{j'}x_{k}x_{k'})\Big]
\in \mathrm{Nil}(KO _{\bar{0}}),
\end{eqnarray*}
where $h\in KO _{1}\cap KO _{\bar{0}}.$ By Lemma \ref{l2}(2), we have
$a_{j}=0.$ This contradicts the assumption that $a_{j}\neq0.$ Thus
$\mathfrak{M}\subseteq KO _{0}\cap KO _{\bar{1}}$ and the proof is complete.
\end{proof}

The key in this paper is the following proposition.

\begin{proposition}\label{t1}
$KO_{0}$ is an invariant subalgebra of $KO $.
\end{proposition}
\begin{proof}
It follows from Propositions \ref{p12} and \ref{p7}.
\end{proof}

\section{Filtration, automorphisms and classification}

One of the main results is as follows, which is a direct consequence of  Proposition \ref{t1} and the following
Lemmas \ref{l13} and
\ref{p14}.
\begin{theorem}\label{t15}
The principal filtration of $KO $ is invariant under the
automorphisms of $KO,$ that is, $\varphi(KO_i)=KO_i$ for all $i\geq-2$ and all automorphisms $\varphi$ of $KO.$\qed
\end{theorem}

\begin{lemma}\label{l13}
$KO _{-1}/KO _{0}$ is the unique irreducible $KO _{0}$-submodule of
$KO /KO _{0}$. In particular, $KO _{-1}$ is an invariant subalgebra
of $KO $.
\end{lemma}
\begin{proof}
Clearly, $KO _{-1}/KO _{0}$ is an irreducible $KO _{0}$-submodule.
To show the uniqueness, suppose $M/KO _{0}$ is a nonzero $KO _{0}$-submodule of $KO /KO_{0},$
where $M\supset KO _{0}$ is a $KO _{0}$-submodule of $KO.$
For $0\neq y\in M,$ one may write $y=D_{KO}(1)+y',$ where $y'\in KO
_{-1}.$ Then
\begin{eqnarray*}
[D_{KO}(x_{i}x_{2n+1}), D_{KO}(1)+y']\in M\quad\mbox{for all\;} i\in \overline{1,2n}.
\end{eqnarray*}
Note that $[D_{KO}(x_{i}x_{2n+1}), y'] \in M.$ We have
$$2D_{KO}(x_{i})=[D_{KO}(x_{i}x_{2n+1}),
D_{KO}(1)]\in M$$ for all $i\in \overline{1,2n}.$ Therefore, $KO
_{-1}/KO _{0}\subseteq M/KO _{0}.$ The proof is complete.
\end{proof}
\begin{lemma}\label{p14}
$KO _{i}=\{y\in KO _{i-1}\mid [y, KO _{-1}]\subseteq KO _{i-1}\} $
for all $i\geq1.$
\end{lemma}
\begin{proof}
Put
$$
\mathfrak{h}_{i}:=\{y\in KO _{i-1}\mid [y, KO _{-1}]\subseteq KO
_{i-1}\}.
$$
Clearly, $KO _{i}\subseteq \mathfrak{h}_{i}.$ For all $y\in
\mathfrak{h}_{i},$ we may write $y:=\sum_{j\geq i-1}y_{j},$ where
$y_{j}\in KO _{[j]}.$ By the definition of $\mathfrak{h}_{i},$
$[y_{i-1},KO _{-1}]=0.$ Let
$y_{i-1}:=\sum_{\alpha,u}b_{\alpha,u}D_{KO}(x^{(\alpha)}x^{u}),$
where $\alpha\in \mathbb{A},$ $u\in \mathbb{B},$
$D_{KO}(x^{(\alpha)}x^{u})\in KO _{[i-1]}\subseteq KO _{0},$
$b_{\alpha,u}\in \mathbb{F}.$ For any fixed $\beta\neq 0$ with
$\beta_k\geq 1$ for some $k\in\overline{1,n}$, we have
\begin{eqnarray*}
0&=&\Big[\sum _{\alpha,u}b_{\alpha,u}D_{KO}(x^{(\alpha)}x^u),
D_{KO}(x_{k^{\prime}})\Big]\\
&=&\sum
_{\alpha,u}b_{\alpha,u}D_{KO}(\partial_k(x^{(\alpha)}x^u)-(-1)^{{\rm
p}(x^u)}\partial_{2n+1}(x^{(\alpha)}x^u)x_{k^\prime}).
\end{eqnarray*}
Consequently,  $b_{\alpha,u}=0$ whenever $\alpha\neq 0$. It remains
to consider the case $\alpha=0$. Fix  any  $v\neq \emptyset.$ If
 there is  $l\in \overline{n+1,
2n}$ such that $l\in v.$ Then
\begin{eqnarray*}
0&=&\Big[\sum _{0,u}b_{0,u}
D_{KO}(x^u),D_{KO}(x_{l^{\prime}})\Big]
\\
&=&\sum
_{0,u}b_{0,u}D_{KO}((-1)^{{\rm p}(x^u)}\partial_l(x^u)-(-1)^{{\rm
p}(x^u)}\partial_{2n+1}(x^u)x_{l^\prime}).
\end{eqnarray*} It follows that
 $b_{0,v}=0,$ where $ \langle 2n+1\rangle\neq v\in \mathbb{B}.$ Taking $i\in \overline{1,2n},$
since
 \begin{eqnarray*}
 0=[b_{0,\langle 2n+1\rangle}D_{KO}(x_{2n+1}), D_{KO}(x_{i})]=b_{0,\langle
 2n+1\rangle}D_{KO}(x_{i}),
 \end{eqnarray*}
 we have $b_{0,\langle
 2n+1\rangle}=0.$   Therefore, $y_{i-1}=0$ and then
  $\frak{h}_{i}\subset KO_i$.
\end{proof}

As a corollary we give a characterization of the automorphisms of $KO$:
\begin{theorem}\label{t17}
Two automorphisms of $KO$ coincide if and only if they coincide on the $-1$ component
$KO_{[-1]}.$
\end{theorem}
\begin{proof}
Let $\phi$ and $\psi$ be the automorphisms of $KO.$ It suffices to
prove that $\phi\mid_{KO_{[-1]}}=\psi\mid_{KO_{[-1]}}$ implies
$\phi=\psi.$ As in \cite[Corollary 18]{lz2}, using Theorem \ref{t15}
one can prove that $\phi\mid_{KO_{-1}}=\psi\mid_{KO_{-1}}.$ Since
$$\phi(D_{KO}(1))=\phi([D_{KO}(x_{1}),D_{KO}(x_{1'})])=
\psi([D_{KO}(x_{1}),D_{KO}(x_{1'})])=\psi(D_{KO}(1)),$$ we have
$\phi\mid_{KO_{[-2]}}=\psi\mid_{KO_{[-2]}}.$ Then $\phi=\psi$ and the
proof is complete.
\end{proof}

We are now in  position to state the final main result in this
paper, which
 says that the parameter $n$ defining the odd Contact superalgebra $KO(n,n+1)$ is  intrinsic and then all the
 infinite-dimensional odd Contact superalgebras are classified up to isomorphisms.
\begin{theorem}\label{t16}
$KO(n,n+1)\cong KO(m,m+1)$ if and only if $n=m.$
\end{theorem}

\noindent {\textit{Proof of Theorem \ref{t16}.}}
One direction is obvious. Assume
that $\sigma: KO(n,n+1)\longrightarrow KO(m,m+1)$ is an isomorphism
of Lie superalgebras. Clearly,
$$
\sigma(\mathrm{Nor}_{KO(n,n+1)_{\bar{0}}}(\mathrm{Nil}(KO(n,n+1)_{\bar{0}})))
=\mathrm{Nor}_{KO(m,m+1)_{\bar{0}}}(\mathrm{Nil}(KO(m,m+1)_{\bar{0}})).
$$
In view of  the proof of Proposition \ref{p12}, we have
$$KO(n,n+1)_{0}\cap
KO(n,n+1)_{\bar{0}}=\mathrm{Nor}_{KO(n,n+1)_{\bar{0}}}(\mathrm{Nil}(KO(n,n+1)_{\bar{0}})).$$
Therefore,
\begin{eqnarray}\label{e1}
\sigma(KO(n,n+1)_{0}\cap KO(n,n+1)_{\bar{0}})=KO(m,m+1)_{0}\cap
KO(m,m+1)_{\bar{0}}.
\end{eqnarray}
Recall  that
$$\mathfrak{Q}= \{y\in KO(n,n+1)_{\bar{1}}\mid
[y,KO(n,n+1)_{\bar{1}}]\subseteq KO(n,n+1)_{0}\cap
KO(n,n+1)_{\bar{0}} \}$$
 and
 $$\mathfrak{M}=\{y\in KO(n,n+1)_{\bar{1}}\mid
[y,\mathfrak{Q}]\subseteq\mathrm{Nil}(KO(n,n+1)_{\bar{0}})\}.
$$
 By  Lemma \ref{p5} and Proposition \ref{p7}, we have
\begin{eqnarray}\label{e4}
\sigma(KO(n,n+1)_{0}\cap KO(n,n+1)_{\bar{1}})=KO(m,m+1)_{0}\cap
KO(m,m+1)_{\bar{1}}.
\end{eqnarray}
By (\ref{e1}) and ({\ref{e4}}), we have
$$
\sigma(KO(n,n+1)_{0})=KO(m,m+1)_{0}.
$$
Therefore $\sigma$ induces an isomorphism of $\mathbb{Z}_{2}$-graded
vector spaces
$$
\sigma': KO(n,n+1)/KO(n,n+1)_0\longrightarrow KO(m,m+1)/KO(m,m+1)_0.
$$
Since $KO(n,n+1)/KO(n,n+1)_0\cong KO(n,n+1)_{[-2]}\oplus
KO(n,n+1)_{[-1]}$
 as $\mathbb{Z}_{2}$-graded vector spaces,
we have
$$\mathrm{dim}(KO(n,n+1)_{[-2]}\oplus KO(n,n+1)_{[-1]})
=\mathrm{dim}(KO(m,m+1)_{[-2]}\oplus KO(m,m+1)_{[-1]}).$$ This
implies that $n+1=m+1,$ that is, $n=m.$ The proof is complete.

\end{document}